\documentclass[authoryear]{article} 

\RequirePackage{amsthm,amsmath}
\RequirePackage[numbers]{natbib}
\RequirePackage[colorlinks,citecolor=blue,urlcolor=blue]{hyperref}

\usepackage{amssymb,amsfonts,color,xcolor} 
\usepackage{mathtools, thmtools}
\usepackage{bbm,bm}
\usepackage[english]{babel}
\usepackage{authblk}
\usepackage{comment}
\usepackage{graphicx}
\usepackage{enumerate}
\usepackage[title,titletoc,toc]{appendix}

\usepackage[a4paper,bindingoffset=0.5cm,left=2cm,right=2cm,top=2.5cm,bottom=2cm,footskip=.8cm]{geometry} 

\newcommand{\vb}{\vspace{3mm}}

\renewcommand{\bar}{\overline}

\allowdisplaybreaks

\newtheorem{theorem}{Theorem}

\newtheorem{proposition}{Proposition}
\newtheorem{lemma}{Lemma}

\theoremstyle{definition}

\newtheorem{assumption}{Assumption}
\newtheorem{remark}{Remark}
\newtheorem{example}{Example}

\begin{document}

\title{Asymptotics and approximations of ruin probabilities\\ for multivariate risk processes in a Markovian environment}

\author[1,2]{G.A. Delsing
\footnote{Corresponding author. E-mail address: G.A.Delsing@uva.nl}}
\author[1,3]{M.R.H. Mandjes}
\author[1,4]{P.J.C. Spreij}
\author[1,2]{E.M.M. Winands}

\affil[1]{Korteweg-de Vries Institute 
University of Amsterdam, Science Park 107, 1098 XH Amsterdam, the Netherlands}

\affil[2]{Rabobank, Croeselaan 18, 3521 CB Utrecht, the Netherlands}

\affil[3]{CWI, Science Park 123, 1098 XG Amsterdam, the Netherlands}

\affil[4]{Radboud University, Heyendaalseweg 135, 6525 AJ Nijmegen, the Netherlands}
\maketitle
\begin{abstract} \noindent This paper develops asymptotics and approximations for ruin probabilities in a multivariate risk setting. We consider a model in which the individual reserve processes are driven by a common Markovian environmental process. We subsequently consider a regime in which the claim arrival intensity and transition rates of the environmental process are jointly sped up, and one in which there is (with overwhelming probability) maximally one transition of the environmental process in the time interval considered. The approximations are extensively tested in a series of numerical experiments.  

\vb

\noindent
\end{abstract}

\section{Introduction}
{Ruin theory} is the branch of applied probability that quantifies a firm's vulnerability to insolvency and ruin. So as to control risks, regulating authorities impose restrictions on the capital reserve that should be minimally kept. For instance, insurance companies should manage their capital reserve level such that the probability of ruin within one year is below a given threshold, e.g.\ 0.01\%. The main interest in ruin theory, and the objective of this paper, is to develop quantitative techniques to assess ruin probabilities. Even though ruin theory originates in actuarial science, and is common in the insurance industry, it also has applications in operational risk modelling \citep*[see][]{Kaishev}, credit risk modelling \citep*[see][]{Ruin_credit}, and related fields. 
This paper's distinguishing feature is that we succeed in analyzing a multivariate ruin problem, where the individual components are driven by a {\it common} environmental process.

\vb

In ruin theory, technical ruin of a firm occurs when the surplus (i.e., the capital reserve level) of the firm drops below 0. The evolution of the surplus of the firm over time experiences fluctuations due to amounts claimed and premiums earned. It depends on various characteristics including 
the distribution of the claim amounts, the inter-arrival time of claims, and the incoming premiums. In this paper we study a multi-dimensional variant of the classical Cram\'{e}r-Lundberg model \citep*[see][]{lundberg1903approximerad}.
Initially, the focus of ruin theory was on the probability $\phi(u)$ of \textit{ultimate ruin}, i.e.\ the probability that the surplus {\it ever} drops below zero given the initial reserve $u$. Later these results have been extended in many ways, most notably (i)~ruin in finite time, (ii)~other claim arrival processes than the Poisson process, and (iii)~asymptotics of $\phi(u)$ for $u$ large. We refer to e.g.\ \citet*{MR2766220} for a detailed account.  

\vb

While most of the existing literature primarily considers the univariate setting describing a single reserve process, in practice firms often have multiple lines of 
business. As a consequence, it is relevant to consider the probability of the surplus of one or more of the business lines dropping below zero. In a multivariate risk model the individual capital surplus processes are typically affected by common environmental factors (think of the impact of the weather on health insurance and agriculture insurance). It is noted that such a multivariate risk model has obvious applications in related settings, such as credit risk.

The main contributions of the paper are the following. We work with a simple yet versatile
multivariate risk model, in which the components are made correlated (by making them
dependent on a common Markovian environmental process). 
The focus is on developing techniques to assess ruin probabilities in this multivariate setting. 
We distinguish between two regimes, corresponding to the speed at which the environmental process evolves. 
For these regimes we derive asymptotic results for the multivariate risk process, leading to closed-form approximations for the multivariate finite-time ruin probability. A thorough numerical study provides us insight into the model dynamics and the impact of its parameters. 

\vb

We proceed with a few more words on the related literature, and its relation to our work. 
Multivariate risk processes play a prominent role in various studies (see e.g.\ the overview \citet*[Ch. XIII.9]{MR2766220}), but capturing the corresponding joint ruin probability has proven challenging \citep*[see e.g.][]{MR2322127,MR2016771}.
Our work is inspired by earlier work by \citet*{Loisel2004_unpublished,Loisel2007b}, which also makes use of a Markovian environmental process. The main difference between Loisel's work and the present paper is that the former focuses on an iterative scheme, whereas this paper uses asymptotic results to develop approximations for the finite-time ruin probability. Our work includes an extensive numerical study providing practical and theoretical insights into the model and the suggested approximations.

Markov-modulated risk models have been intensively studied \citep*[see e.g.][]{Asmussen1989,Reinhard1984,Dickson}. The common basic idea is that the claim arrivals and claim amounts are influenced by an external environment process. For the univariate finite-time ruin probability in a Markovian environment no closed-form expression is available (except for very special cases), but various approximations and efficient simulation approaches have been developed; a comparison of the approximations was presented by \citet*{Asmussen1989}.
This paper complements these works in the sense that we now consider their multivariate counterpart. 

\vb

This paper is organized as follows. Section \ref{sec_model} presents the model and preliminaries. It defines the risk process for each individual business line, and presents the multivariate risk model by introducing the Markovian environmental process and the multivariate finite-time ruin probability.
 
In Sections \ref{sec_diff} and \ref{sec_single_switch}, respectively, a diffusion approximation and a \textit{single-switch} approximation of the multivariate risk process are presented. The diffusion approximation corresponds to a regime in which the claim arrival process and the environmental process are sped up (but in a coordinated manner); under this scaling weak convergence to an appropriate multivariate Brownian motion is established. For such multivariate Brownian motions there is a considerable amount of literature on first passage probabilities; in particular, the bivariate case even allows for an explicit calculation of the finite-time ruin probability.
Diffusion approximations tend to work well in e.g.\ scenarios with a relatively large amount of outgoing claims over a finite horizon. 

The single-switch approximation considers a regime in which, with overwhelming probability, the background process has either zero or one transition in the time interval considered. This approximation is particularly accurate in case the time horizon is relatively short in comparison with the speed of the environmental process. For example, regulatory requirements for the insurance and banking industry are based on a 1-year horizon while business cycles (corresponding with the environmental process) usually last multiple years (average of 8 years). 

Numerical examples are provided in Section \ref{sec_results}. By considering different parameter settings, the single-switch approximation turns out to perform well in case of low environmental transition rates while the diffusion approximation is a more suitable choice when both the arrival intensities as well as the transition rates are high.  
Section \ref{sec_conclusion} concludes this paper, and discusses areas for future research. A number of technical results are collected in the Appendix.

\section{A multivariate risk model}\label{sec_model}
In this section, the multivariate risk process is introduced for a firm with multiple business lines. Each business line has its own reserve process.
For each of the business lines, there is some initial capital reserve, which increases due to premiums (that come in at a fixed rate per unit time), and decreases due to claims. We focus on the probability that, given certain initial reserve levels, one or more of the reserve processes (corresponding to the business lines) drop below zero before a specified time $T>0$. We assume no impact of insolvency of one business line on the others, and each line of business is free of expenses, taxes, and commissions. 

We model the multivariate (of dimension $m\in{\mathbb N}$) risk setting by a multi-dimensional variant of the classical Cram\'{e}r-Lundberg model. For the sake of clarity, we first recall this `base model', which later in this section will be extended so as to include dependencies between its components. 

The dynamics of the $m$-dimensional Cram\'{e}r-Lundberg reserve process ${\boldsymbol X}:=(X_1,\ldots,X_m)$, with $X_i:=\{X_i(t):t\geq 0\}$, are as follows. 
\begin{itemize}
\item[$\circ$]
Let $u_i\ge 0$ be the initial reserve of component $i\in\{1,\ldots,m\}$. 
\item[$\circ$]
Component $i$ grows linearly due to premiums at a constant premium rate $r_i\geq 0$ per unit time. 
\item[$\circ$]
Let ${\boldsymbol A}=(A_1,...,A_m)$ be the $m$-dimensional claim arrival process, where we assume that the arrival process $A_i:=\{A_i(t):t\geq 0\}$ is a Poisson process with parameter $\lambda_i$; the arrival times are denoted by $\{\tau_{i,k}:k\in\mathbb{N}\}$.

\noindent
The claim sizes of component $i$, $\{Z_{i,k}:k\in\mathbb{N}\}$, form a sequence of i.i.d.\ random variables independent of ${\boldsymbol A}$, with finite mean $\mu_i$, finite variance $\sigma_i^2$ and distribution function $F_i(\cdot)$. 
\end{itemize}
Combining the above components, the capital surplus process $X_i(t)$ for component $i$ equals
 \begin{equation}\label{eq_ruin}
X_i(t):=u_i+r_i t - \sum_{k=1}^{A_i(t)}Z_{i,k}.
\end{equation}
In Section \ref{subsec_Markov} we introduce Markov-modulation: the claim arrival intensities $\lambda_i$ and claim size distribution functions $F_i(\cdot)$ are not fixed in time but depend on an underlying environmental Markov process. Importantly, this environmental process is the same for all $m$ components (and could for instance reflect the `state of the economy', or weather conditions), thus rendering the individual risk processes $X_i$ dependent. In Section \ref{subsec_ruin} we define the ruin probabilities that we focus on in this paper. 

\subsection{Environmental dependence}\label{subsec_Markov}
In this subsection we point out how we make the $m$ components $X_i$ dependent and allow for the individual claim processes to change over time by using a Markovian environmental process. 
This environmental state process, denoted by $J:=\{J(t):t\geq 0\}$, is a Markov process with finite support $S=\{1,...,I\}$. Let, as before, ${\boldsymbol A}=(A_1,...,A_m)$ be the joint claim arrival process. We let the arrival rate pertaining to $A_i$ be $\lambda_{i,j}$ when the environmental process $J$ is in state $j$. As before, the corresponding arrival times are denoted by $\{\tau_{i,k}:k\in\mathbb{N}\}$.

Again we let $\{Z_{i,k}:k\in\mathbb{N}\}$ be the $k$-th claim size of component $i$ (i.e., arriving at $\tau_{i,k}$). If at $\tau_{i,k}$ the environmental process $J$ is in state $j$, the claim size is sampled (independently of anything else) from a distribution that has distribution function $F_{i,j}(\cdot)$; in other words,
\[\mathbb{P}(Z_{i,k}\leq x\,|\,J(\tau_{i,k})=j)=F_{i,j}(x),\:\: \ x\ge 0.\]
We assume that for each $i\in\{1,...,m\}$ and $j\in S$, the distribution of the claim sizes  has finite variance; let $\mu_{i,j}$ and $\sigma_{i,j}^2$ be the corresponding mean and variance.

The joint reserve process is again given by (\ref{eq_ruin}), but now with the components' arrival and claim size processes affected by the common environmental process $J$, in the way described above. Observe that the processes $X_i$ are dependent, but {\it conditional on} the the path of the environmental state process they are not. 

The transition rate matrix governing $J$ is denoted by $Q=(q_{k,l})_{k,l\in A}$ with $q_k:=-q_{k,k}=\sum_{l\neq k}q_{k,l}$ and initial environmental state distribution denoted by $p=(p_1,...,p_I)$ where $p_j:=\mathbb{P}(J(0)=j)$. Assume $\pi:=(\pi_1,...,\pi_I)$ to be the (unique) stationary distribution of $J$, and $\Pi$ be a matrix with each row being the steady-state vector $\pi$. The so-called {\it fundamental matrix} $\Upsilon=(\Upsilon_{k,l})_{k,l=1,...,I}$ is given by $\Upsilon=\left(\Pi-Q\right)^{-1}-\Pi$.

We use the notation ${\boldsymbol r}:=(r_1,...,r_m)$ and ${\boldsymbol u}:=(u_1,...,u_m)$ to denote the $m$-dimensional vectors of parameters corresponding to the premium rates and initial reserve levels, respectively. 

\subsection{Multivariate ruin and survival probabilities}\label{subsec_ruin}
In this subsection we further detail the quantities we wish to analyse, namely  multivariate ruin probabilities.
For a univariate risk process, say that corresponding to component $i$, the probability of ruin before time $T$ (starting in environmental state $j$ at time 0) is given by
\begin{equation}\label{eq_psi_j}
\psi_i^j(u_i,T):=\mathbb{P}\left(\inf_{t\in[0,T]}X_i(t)<0\,\big|\,X_i(0)=u_i, J(0)=j\right).
\end{equation}
Sometimes we are also interested in the complementary probability, referred to as the {\it survival probability}:
\[\bar{\psi}^j_i(u_i,T):=1-\psi_i^j(u_i,T).\]
The univariate ruin probability in a Markov-modulated setting has been intensively studied. Thus far, no explicit solution has been found (except for special cases), but there are various in-depth accounts of  available approximations (such as diffusion approximations, Segerdahl approximations, and corrected diffusion approximations); for an overview we refer to e.g.\ 
\citet*{Asmussen1989}.

\vb

The multivariate ruin probability (of {\it all} $m$ components), assuming $J$ is in state $j$ at time zero, is defined by
\begin{equation}\label{eq_multi_ruin}
\psi^j({\boldsymbol u},T):=\mathbb{P}\left(\inf_{t\in[0,T]} {\boldsymbol X}(t)\prec 0\,\big| \,{\boldsymbol X}(0)={\boldsymbol u}, J(0)=j\right);
\end{equation}
here the inequality `$\prec$' and the infimum are understood to be taken in a component-wise manner. 
The corresponding survival probability is defined in the obvious way. In addition, other performance metrics, such as the probability of ruin of a subset of components, can be defined in a similar fashion.

 \section{Diffusion approximation}\label{sec_diff}
This section presents a diffusion limit of multivariate risk processes introduced in the previous section. This asymptotic result is then used to develop a diffusion approximation of the multivariate finite-time ruin probabilities. The derivation of the diffusion approximation of ruin (and survival) probabilities for the multivariate case follows the one-dimensional case \citep*{ASMUSSEN1987301,Asmussen1989} to the extent that the (in our case multivariate) centered Markov-modulated claims process is approximated by a (in our case multivariate) Brownian motion;
this approximation is based on the functional central limit theorem (FCLT) presented in Section \ref{subsec_FCLT}.
Our multivariate ruin probability $\psi^j({\boldsymbol u},T)$ can then be approximated by its Brownian counterpart, which can be expressed in closed form relying on results for first-passage probabilities for multivariate Brownian motion; see Sections  \ref{subsec_BMextrema}--\ref{subsec_diffusion}.

\subsection{Multivariate FCLT for Markov-modulated compound Poisson processes}\label{subsec_FCLT}
Define the Markov-modulated compound Poisson process ${\boldsymbol Y}:=(Y_1,\cdots,Y_m)$ 
by
\[Y_i(t) :=\sum_{k=1}^{A_i(t)} Z_{i,k},\]
where the processes $A_i$ and $Z_{i,k}$ are generated by the model with environmental dependence, as was introduced in Section \ref{sec_model}. For the case $m=1$,  \citet*{Pang} prove a FCLT under appropriate scalings of the claim arrival process, claim sizes and the underlying Markov environmental process. The main objective of this subsection is to extend this FCLT to the multivariate case. After having introduced some notation, we present the multivariate counterpart of  \citet*[Thm. 1.1]{Pang} in our Theorem \ref{Th_Pang_multi}.

\vb

We define $\mathbb{D}^m$ as the space of $m$-vectors of c\`{a}dl\`{a}g functions which is endowed with the metric \[d^m_B(x,y):=\max_{1\leq j\leq m}d_B(x_j,y_j),\] 
where $d_B$ denotes the Billingsley metric as in \citet*[Sec.\ 28.3]{davidson1994stochastic} and $d^m_B$ induces the product topology. In $\mathbb{D}$ the Skorokhod $J_1$ metric is equivalent to the Billingsley metric, and the separability of $(\mathbb{D},J_1)$ implies both separability of $(\mathbb{D}^m,d^m_B)$ and that $\mathcal{B}_D^m=\mathcal{B}_D\times\mathcal{B}_D\times\cdots\times\mathcal{B}_D$ is the Borel field of $(\mathbb{D}^m,d^m_B)$. The finite-dimensional sets of $\mathbb{D}^m$ is the field generated by the product of $m$ copies of the finite-dimensional sets of $\mathbb{D}$.

\vb

In our scaling, we scale both the claim sizes and the arrival rates. The scaling parameter is $n$, which we let grow large. 
\begin{itemize}
\item[$\circ$] The claim sizes $\{Z_{i,k}^n\}$ have means $\mu_{i}^n$ and variances $\sigma_{i}^n$. It assumed that these are such that, as $n\rightarrow\infty$, $\mu_{i}^n\rightarrow \mu_i$ and $\sigma_{i}^n\rightarrow \sigma_i$.
\item[$\circ$] Now consider the Markovian environmental process 
$J^n$. The corresponding transition rates are sped up by a factor $n^{\alpha}$, with $\alpha>0$, meaning that they become $q_{i,j}^n:=n^{\alpha}q_{i,j}$. The value of $\alpha$ indicates the speeding effect of the modulating environmental process. 
\item[$\circ$]
For component $i\in\{1,\cdots,m\}$ the arrival rates $\lambda_{i}:=(\lambda_{i,1},\cdots\lambda_{i,I})$ are sped up (essentially) linearly. More concretely, for $n\rightarrow \infty$, we assume 
\[\frac{\lambda_{i}^n}{n}\rightarrow \lambda_{i}.\] 
\end{itemize}
We let ${\boldsymbol Y}^n:=(Y_1^n,...,Y_m^n)$ denote a sequence of Markov-modulated compound Poisson processes corresponding to the parameters $\lambda_{i,j}^n$, $\mu_{i,j}^n$ and $\sigma_{i,j}^n$ and the transition rate matrix $Q^n:=(q_{i,j}^n)_{i,j\in S}$. 
Similar notation is used for the process ${\boldsymbol A}^n$.

\vb

Similar to the approach followed by \citet*{Pang}, we define the diffusion-scaled (i.e., centered and normalized) process 
${\hat{\boldsymbol Y}}^n=(\hat{Y}_1^n,...,\hat{Y}_m^n)$, with $\hat{Y}_i^n:=\{\hat{Y}_i^n(t):t\geq 0\}$ defined through
$$\hat{Y}_i^n(t):=\frac{1}{n^\delta}\left(Y_i^n(t)-\bigg(\sum_{j=1}^I\lambda_{i,j}^n\mu_{i,j}\pi_j\bigg) t\right), \ \text{for} \ \frac{1}{2}\leq\delta < 1 \ \text{and} \ t\geq 0.$$

Theorem \ref{Th_Pang_multi} below is the multivariate version of the FCLT for univariate Markov-modulated compound Poisson processes (as was presented by \citet*{Pang}). The full proof is given in Appendix \ref{appA}. 

Before stating the theorem, we first introduce some notation. Let
\begin{align*}\bar{\lambda}_i:=&\,\sum_{j=1}^I\lambda_{i,j}\mu_{i,j}\pi_j, \:\:\: \bar{m}^2_i:=\,\sum_{j=1}^I\lambda_{i,j}\mu_{i,j}^2\pi_j,\\
 \ \bar{\sigma}_i^2:=&\,\sum_{j=1}^I\lambda_{i,j}\sigma_{i,j}^2\pi_j, \:\:{\ \rm and}\:\:\bar{\beta}_{i,j}:=\,2\sum_{k=1}^I\sum_{l=1}^I\lambda_{i,k}\lambda_{j,l}m_{i,k}m_{j,l}\pi_{k}\Upsilon_{k,l}.\end{align*}
 In addition, let $\bar \Sigma_1 :={\rm diag}\{\bar{\sigma}_{1}^2,\ldots,\bar{\sigma}_{m}^2\}$, $\bar \Sigma_2 :={\rm diag}\{\bar{m}_{1}^2,\ldots,\bar{m}_{m}^2\}$, and $\bar\Sigma_3:=(\bar\beta_{i,j})_{i,j\in{S}}.$
\begin{theorem}\label{Th_Pang_multi}
Under the scaling considered, as $n\to\infty$,
$$\hat{\boldsymbol Y}^n\Rightarrow \bar{\boldsymbol Y} \ \text{in \ } (\mathbb {D}^m,d^m_B),$$
where $\bar{Y}$ is a $m$-dimensional Brownian motion with mean zero and covariance matrix 
$$\bar{\Sigma}:=\begin{cases} \bar{\Sigma}^1+\bar{\Sigma}^2, & \text{if \ } \delta=\tfrac{1}{2}, \ \alpha>1, \\
\bar{\Sigma}^1+\bar{\Sigma}^2+\bar{\Sigma}^3, & \text{if \ } \delta=\tfrac{1}{2}, \ \alpha=1, \\
\bar{\Sigma}^3, & \text{if \ } \delta=1-\tfrac{\alpha}{2}, \ \alpha\in(0,1).
\end{cases}$$ 
\end{theorem}
Theorem \ref{Th_Pang_multi} shows that a centered and appropriately scaled version of the multivariate reserve process converges to a multivariate Brownian motion.

\subsection{First passage probabilities of multivariate Brownian motion}\label{subsec_BMextrema}
This subsection reviews results on first-passage probabilities corresponding to multivariate Brownian motions (or, equivalently, the joint distribution of the extrema of multivariate Brownian motions). We describe some of the existing results which we appeal to in the next subsection to derive a diffusion approximation of the multivariate ruin probability $\psi^j({\boldsymbol u},T)$. 

In general, the problem of identifying the
joint distribution of extrema of correlated Brownian motions is highly challenging; solutions are typically of an implicit nature.
\begin{itemize}
\item[$\circ$]
In a pioneering paper, \citet*{Iyengar} considered an analytical expression for the first passage time of a bivariate Brownian motion with zero-drift. 
\item[$\circ$]
For the case of non-zero drift, as is the case in our situation, a double integral expression for the joint cumulative distribution of the suprema of two correlated Brownian motions was later found by \citet*[Lemma 3]{he1995double}. 
\item[$\circ$]
A summary of existing results on the first-passage problem of two-dimensional Brownian motions was recently published by \citet*{Kou2016}.
\item[$\circ$]
For dimensions larger than 2, explicit formula for the joint distribution of extrema of Brownian motions exist only for a specific set of correlations; see e.g.\ \citet*{Escobar}. A multivariate approximation of the joint survival probability as a linear combination of the bivariate survival probabilities has been developed by \citet*[Eqn. (30)]{Bhansali}. 
\end{itemize}

\subsection{Diffusion approximation of the multivariate ruin probability in finite time}\label{subsec_diffusion}
This subsection combines the asymptotic result of Subsection \ref{subsec_FCLT} and the literature study presented in Subsection \ref{subsec_BMextrema} to derive a diffusion approximation of the multivariate finite-time ruin probability.

\vb 

We consider the following specific scaling (obeying the assumptions imposed earlier).  For ease we let the claim sizes $\{Z_{i,k}^n\}$ be independent of $n$. In addition, we assume the arrival rates to be linear in $n$: $\lambda_i^n=n\lambda_i$; define $\bar{\lambda}_i^n$ as $n\bar\lambda_i$. Finally, we consider the setting 
$\delta=\frac{1}{2}$ and $\alpha=1$ (which essentially covers the two other cases), so that $q_{i,j}^n:=n q_{i,j}$.

Let $\bar{{\boldsymbol Y}}$ denote a $m$-dimensional Brownian motion with covariance matrix $\bar{\Sigma}:=\bar{\Sigma}^1+\bar{\Sigma}^2+\bar{\Sigma}^3$, as defined before. 
Scale ${\boldsymbol u}^n = \sqrt{n}\,{\boldsymbol u}$ for a componentwise positive vector ${\boldsymbol u}$ (not depending on $n$), and ${\boldsymbol r}^n = \bar{\boldsymbol \lambda}^n + \sqrt{n}\,{\boldsymbol \varrho}$ for a vector ${\boldsymbol \varrho}$ (not depending on $n$). 
Consider the scaled risk process ${\boldsymbol X}^n$, defined by
\[{\boldsymbol X}^n (t)= {\boldsymbol u}^n +{\boldsymbol r}^n - {\boldsymbol Y}^n(t),\] and the corresponding ruin probability. Then, for any $n$, 
\begin{align*}
\mathbb{P}\left(\inf_{t\in[0,T]}{\boldsymbol X}^n(t)\prec{\boldsymbol 0}\,|\,{\boldsymbol X}^n(0)={\boldsymbol u}^n\right)&=\mathbb{P}\left(\sup_{t\in[0,T]}{\boldsymbol Y}^n(t)-{\boldsymbol r}^nt\succ {\boldsymbol u}^n\right)\\
&=\mathbb{P}\left(\sup_{t\in[0,T]}\frac{{\boldsymbol Y}^n(t)-\bar{{\bm \lambda}}^nt}{\sqrt{n}}-{\boldsymbol \varrho}\,t\succ {{\boldsymbol u}}\right),
\end{align*}
which, by Theorem \ref{Th_Pang_multi} in combination with the continuous mapping theorem, as $n\to\infty$ converges to
\[\mathbb{P}\left(\sup_{t\in[0,T]}\bar{\boldsymbol Y}(t)-{\boldsymbol \varrho}\,t\succ {{\boldsymbol u}}\right).\]

For practical purposes the last probability can be evaluated using the results found in the literature as summarized in the previous subsection. For the bivariate case, the numerical work performed in Section \ref{sec_results} makes use of \citet*[Lemma 3]{he1995double}.
Similar to the univariate case, the multivariate diffusion approximation gains accuracy when increasing the claim arrival intensities $\lambda$ and the transition rates $Q$. To explore the approximation's accuracy,  we perform a numerical study in Section \ref{sec_results}.

\begin{remark}Observing that our scaling of the $\lambda_i$ and $q_{i,j}$ effectively means that we scale time by a factor $n$, we remark that the above scaling is equivalent to the {\it heavy-traffic regime} in queueing theory; cf.\ \cite[Section 7c]{MR2766220}. To see this, recall that in such a heavy-traffic scaling, the drift of the queueing process is some $-\varepsilon$, whereas time is blown up by a factor $1/\varepsilon^2$; under this scaling the queue's workload process multiplied by $\epsilon$ weakly converges to reflected Brownian motion as $\varepsilon\downarrow 0.$ 
\end{remark}

\section{Single-switch approximation}\label{sec_single_switch}
This section presents a single-switch approximation of a multivariate risk process and derives the corresponding approximation of the finite-time multivariate ruin probability.
The single-switch approximation corresponds to a regime in which, with high probability, the environmental process has either zero or one transition in the time interval considered.  As a result, the approximation is expected to work well in case the environmental process is slow relative to the claim arrival process.

In the single-switch approximation of the multivariate ruin probability a crucial role is played by the time $\tau$, being the time of the first switch of the environmental state. Given the initial state of the environmental process is $j$, for each new state $k\in S\backslash\{j\}$ we define the corresponding single-switch multivariate ruin probability by
\begin{equation}\label{eq_app_switch}
\psi^{j,k}({\boldsymbol u},\tau ,T):=\mathbb{P}\left(\inf_{t\in[0,T]} {\boldsymbol X}(t)\prec 0\,\Big|\, {\boldsymbol X}(0)={\boldsymbol u},  \{J(t)\}_{0\leq t\leq \tau}=j,\{J(t)\}_{\tau< t\leq T}=k\right);
\end{equation}
the no-switch multivariate ruin probability is
\begin{equation}
\psi^{j}({\boldsymbol u},T):=\mathbb{P}\left(\inf_{t\in[0,T]} {\boldsymbol X}(t)\prec 0\,\Big|\, {\boldsymbol X}(0)={\boldsymbol u},  \{J(t)\}_{0\leq t\leq \tau}=j\right).
\end{equation}
 \normalsize
In line with notation used earlier, $\psi^{j,k}_i(u_i,\tau,T)$ denotes the single-switch ruin probability pertaining to component $i$,
and $\psi^{j}_i(u_i,\tau,T)$ is its no-switch counterpart. When $\tau> T$, the transitioned environmental state $k$ can be omitted as it does not influence the probability. Conditional on the time of the single switch of the environmental state, the risk processes of the individual business lines are independent.

\vb

As a result of the Markov property of the environmental state process, the environmental state switches from starting state $j$ to $k$ after an exponentially distributed time with intensity $q_{j,k}$. Given the initial environmental state is $j$, the probability of at most 1 switch of the environmental state over time horizon $T$ is thus given by
\begin{equation}\label{eq_total_prob}
P^j(T):=\sum_{k\in S\setminus\{j\}}\int_{0}^T q_{j,k}e^{-q_{j}\tau}e^{-q_{k}(T-\tau)}{\rm d}\tau+e^{-q_j T}=\sum_{k\in S\setminus\{j\}}\frac{q_{j,k}}{q_{j}-q_{k}}\left(e^{-q_{k}T}-e^{-q_{j} T}\right)+e^{-q_j T}.
\end{equation} 
The above gives rise to the following single-switch approximation of the multivariate ruin probability which, for initial environmental state $j$:
$$\chi^{j}({\boldsymbol u},T):=\frac{1}{P^j(T)}\left(\sum_{k\in S\setminus\{j\}}\int_{0}^T\prod_{i=1}^m\psi_i^{j,k}({\boldsymbol u},\tau,T)q_{j,k}e^{-q_{j}\tau}e^{-q_{k}(T-\tau)}{\rm d}\tau+\psi_i^{j}({\boldsymbol u},T)e^{-q_j T}\right).$$ 
The difficulty in computing the approximation lies in determining the single-switch univariate ruin probabilities $\psi_i^{j,k}(u_i,\tau,T)$ and their no-switch counterparts $\psi_i^{j}({\boldsymbol u},T)$.

\subsection{Single-switch univariate ruin probabilities: a BM scaling approximation}\label{subsec_BMscaling}
In this section we propose an easy-to-implement approximation of the single-switch univariate ruin probability $\psi_i^{j,k}({\boldsymbol u},\tau,T)$ and its no-switch counterpart $\psi_i^{j}({\boldsymbol u},T)$. We do so
by conditioning on the capital surplus level at the time of the switch $\tau$. More specifically, for business line $i$, starting in environmental state $j$ and transitioning to state $k$ at time $\tau$, 
\begin{align}\label{eq_def_switch}
\psi_i^{j,k}({u}_i,\tau,T):&=\mathbb{P}\left(\inf_{t\in[0,T]} X_i(t)<0\,\Big| \,X_i(0)=u_i,  \{J(t)\}_{0\leq t\leq \tau}=j,\{J(t)\}_{\tau<t\leq T}=k\right)\nonumber\\
&=\psi_i^{j}(u_i,\tau)+\bar{\psi}_i^{j}(u_i,\tau)\int_0^\infty\psi_i^{k}(v,T-\tau)\,{\rm d}G_{i,u_i}^{j,\tau}(v),
\end{align}
where $G_{i,u_i}^{j,\tau}(\cdot)$ denotes the probability distribution function of $X_i(\tau)$ with $X_i(0)=u_i$ conditional on no ruin having occurred up to time $\tau$, i.e.,
$$G_{i,u_i}^{j,\tau}(v):=\mathbb{P}\left(X_i(\tau)\leq v\,\Big| \,X_i(0)=u_i,  \{J(t)\}_{0\leq t\leq \tau}=j,\inf_{t\in[0,\tau]}X_i(t)>0\right).$$
Without conditioning on no ruin occurring prior to time $\tau$ (i.e., the requirement that $X_i(t)>0$ be positive for all $t\in[0,\tau]$), the above probability distribution of the capital surplus could have been computed relying on calculations involving compound Poisson processes; the fact that we have to incorporate this condition complicates the analysis considerably.
We propose the following workaround. 
\begin{itemize}
\item[$\circ$]
We introduce a Brownian-motion based scaling factor, which helps us approximate the conditional probability distribution $G_{i,u_i}^{j,\tau}(\cdot)$.
To this end, we first define
\begin{equation}
\Xi^{j,\tau}_{i,u_i}(v):=\frac{\mathbb{P}\left(X_i(\tau)\leq  v \,\Big|\, X_i(0)=u_i, \ \{J(s)\}_{0\leq s\leq \tau}=j, \ \inf_{s\in [0,\tau]} X_i(s)>0 \right)}{\mathbb{P}\left(X_i(\tau)\leq v\, \Big|\, X_i(0)=u_i, \ \{J(s)\}_{0\leq s\leq \tau}=j\right)}.
\end{equation}
We thus find that
$$G_{i,u_i}^{j,\tau}(v)=\Xi^{j,\tau}_{i,u_i}(v)
\times \mathbb{P}\left(X_i(\tau)\leq v\, \Big|\, X_i(0)=u_i, \ \{J(s)\}_{0\leq s\leq \tau}=j\right).$$
Similarly we define the density of the conditional probability distribution $g_{i,u_i}^{j,\tau}(v)$ (assumed to exist) and the corresponding scaling factor $\xi^{j,\tau}_{i,u_i}(v)$.
\item[$\circ$]
The next step is to approximate the univariate risk process $X_i(\cdot)$ by a Brownian motion $B_{i,j}(\cdot)$; the drift and variance coefficient are chosen so as to match the first two moments. 
Conditional on $\{J(s)\}_{0\leq s\leq \tau}=j$, the process $X_i(\cdot)$ over the time interval $[0,\tau]$ can be approximated by a Brownian motion with drift $r_i-\lambda_{i,j}\mu_{i,j}$ and variance $\lambda_{i,j}(\sigma_{i,j}^2+\mu_{i,j}^2)$, applying the results from e.g.\ \citet*{Grandell1977}. 
\item[$\circ$]
The scaling factor $\Xi^{j,\tau}_{i,u_i}(v)$ can then be approximated by its Brownian counterpart 
\[\hat\Xi^{j,\tau}_{i,u_i}(v):=\frac{\mathbb{P}\left(B_{i,j}(\tau)\leq  v \,\Big|\, B_{i,j}(0)=u_i, \inf_{s\in [0,\tau]} B_{i,j}(s)>0 \right)}{\mathbb{P}\left(B_{i,j}(\tau)\leq v\, \Big|\, B_{i,j}(0)=u_i\right)}.\]
Given that we know, conditional on its values at times $0$ and $\tau$, the distribution of the maximum of a Brownian motion in the interval $[0,\tau]$ (relying on explicit Brownian-bridge calculations), we can explicitly evaluate this expression. Along the same lines, an approximation of $\xi^{j,\tau}_{i,u_i}(v)$) can be found; this can be done using results for first-passage probabilities for Brownian motion as can be found in \citet*[Cor. 8.1 \& Cor. 8.2]{Joshi2003}. We thus find\begin{equation}
\hat{\xi}^{j,\tau}_{i,u_i}(v)=\frac{\displaystyle 1-\exp\bigg({-\frac{4u_i(u_i+v)}{2\lambda_{i,j}(\sigma_{i,j}^2+\mu_{i,j}^2)\tau}}\bigg)}{\displaystyle N\left(\frac{(r_i-\lambda_{i,j}\mu_{i,j})\tau+u_i}{\sqrt{\lambda_{i,j}(\sigma_{i,j}^2+\mu_{i,j}^2) \tau}}\right)-\exp\bigg({\frac{-2(r_i-\lambda_{i,j}\mu_{i,j}) u_i}{\lambda_{i,j}(\sigma_{i,j}^2+\mu_{i,j}^2)}}\bigg)N\left(\frac{-u_i+(r_i-\lambda_{i,j}\mu_{i,j}) \tau}{\sqrt{\lambda_{i,j}(\sigma_{i,j}^2+\mu_{i,j}^2) \tau}}\right)},
\label{eq_SF}
\end{equation}
where $N(\cdot)$ represents the cumulative normal distribution function.
\end{itemize}
Putting all above components together, we can now approximate $\chi^j({\boldsymbol u},T).$
In Example \ref{ex_base} in Section \ref{sec_results} the (numerical) performance of the resulting single-switch approximation is investigated in detail.
\subsection{Special case: Exponential claims with environment-independent claim sizes}\label{subsec_special}
This section considers a specific instance of the risk model introduced in Section \ref{sec_model}: exponentially distributed claims without environmental influence on the claim sizes. In other words, the claim size distribution is independent of the environmental state $j\in S$ and given by $F_{i,j}(x)=1-\exp\{-\frac{x}{\mu_i}\}$, $x\in\mathbb{R}$.
For this special instance of the model, the single-switch ruin probability of component $i$ starting in environmental state $j$ and switching to state $k$ after time $\tau$, $\psi_i^{j,k}(u_i,\tau,T)$, can be calculated explicitly. 
\begin{proposition}\label{prop_exp} For component $i\in\{1,\ldots,m\}$, environmental states $j,k\in S$ and transition time $\tau\in(0,T]$,
$$\psi_i^{j,k}(u_i,\tau,T)=\begin{cases}
 {\displaystyle \frac{\lambda^*_{i}\mu_i}{r_i}} \exp\Bigg\{-\left({\displaystyle \frac{1}{\mu_i}}- {\displaystyle\frac{\lambda^*_{i}}{r_i}}\right)u_i\Bigg\}-{\displaystyle \frac{1}{\pi}}{\displaystyle \int_0^\pi} {\displaystyle \frac{f_1(\theta)f_2(\theta)}{f_3(\theta)}}{\rm d}\theta, & \text{for } r_i>\lambda_{i}\mu_i\\
  &\vspace{-3mm}\\
  1-{\displaystyle \frac{1}{\pi}}{\displaystyle \int_0^\pi} {\displaystyle \frac{f_1(\theta)f_2(\theta)}{f_3(\theta)}}{\rm d}\theta, &\text{for } r_i\leq\lambda^*_{i}\mu_i
\end{cases}$$
\normalsize
where
\begin{align*}
f_1(\theta)&=\frac{\lambda^*_{i}\mu_i}{r_i}\exp\Bigg\{\frac{2 T\sqrt{r_i\lambda^*_{i}}}{\sqrt{\mu_i}}\cos\theta-\left(r_i/\mu_i+\lambda^*_{i}\right)T+\frac{u_i}{\mu_i}\left(\frac{\sqrt{\lambda^*_{i}\mu_i}}{\sqrt{r_i}}\cos\theta-1\right)\Bigg\}\\
f_2(\theta)&=\cos\left(\frac{u_i\sqrt{\lambda^*_{i}}}{\sqrt{r_i\mu_i}}\sin\theta\right)-\cos\left(\frac{u_i\sqrt{\lambda^*_{i}}}{\sqrt{r_i\mu_i}}\sin\theta+2\theta\right)\\
f_3(\theta)&=1+\frac{\lambda^*_{i}\mu_i}{r_i}-2\frac{\sqrt{\lambda^*_{i}\mu_i}}{\sqrt{r_i}}\cos\theta\\
\lambda^*_i&=\frac{\lambda_{i,j}\tau+\lambda_{i,k}(T-\tau)}{T}
\end{align*}
\end{proposition}
\begin{proof}
The proof follows by mimicking the proof of \citet*[Prop.\ 1.3, Ch.\ V]{MR2766220}, and noting that \citet*[Lemma 8.4]{asmussen2003applied} can be extended to a deterministic arrival intensity using $\lambda^*=T^{-1}\cdot \int_0^T\lambda(t){\rm d}t$. 
Furthermore, the case $\mu_i \neq 1$ can be deduced from the case $\mu_i=1$ via (in self-evident notation)
\[\psi^{j,k}_{i,\lambda_{i},1/\mu_i}(u_i,\tau,T)=\psi^{j,k}_{i,\lambda_{i}\mu_i,1}\left(\frac{u_i}{\mu_i},\frac{\tau}{\mu_i},\frac{T}{\mu_i}\right).\]
Likewise, the case $r_i\neq 1$ follows from \[\psi^{j,k}_{i,\lambda_{i},r_i}(u_i,\tau,T)=\psi^{j,k}_{i,\lambda_{i}/r_i,1}(u_i,r_i \tau,r_i T).\]
This proves the claim.
\end{proof}

\section{Numerical Results}
\label{sec_results}
This section serves as a numerical illustration of the diffusion and single-switch approximations of the finite-time multivariate ruin probability, as presented in Sections \ref{sec_diff} and \ref{sec_single_switch}, respectively. To this end we consider the same setup as used by \citet*{Asmussen1989} for the univariate Markov-modulated risk process. 
We present three representative examples in which we investigate for which parameter settings the approximations perform well.
The examples assess:
\begin{itemize}
\item[$\circ$] the impact of correlation on the multivariate ruin probabilities for the base parameter setting;
\item[$\circ$] the impact of the transition rates of the environmental process; and
\item[$\circ$] the impact of the claim arrival intensities on the performance of both approximations.
\end{itemize}
All computations were performed in R using adaptive quadrature methods for numerical integration. We have performed extensive additional experimentation, but the phenomena observed there do not provide any extra insights relative to those obtained from the three reported examples.

\begin{example}(\textit{Base parameter settings and impact of correlation.})\label{ex_base}
This example illustrates the performance of the diffusion and single-switch approximation of the multivariate ruin probability as given in Equation \eqref{eq_multi_ruin}, as a function of the time horizon $T$. 

We consider two lines of business (i.e., $m=2$) with the same risk profile, i.e.\ the same parameters in the risk process. There are two environmental states (i.e., $I=2$) such that state 1 represents a `booming' state of the economy with a low arrival rate of outgoing claims $\lambda_{1,1}=\lambda_{2,1}=0.45$, whereas state 2 represents a `recession' with a relatively large amount of outgoing claims $\lambda_{1,2}=\lambda_{2,2}=1.8$. Transitions from state~1 to state 2 of $J(t)$ occur at rate $1$, those state 2 to state 1 at rate $2$. The claim size distribution $F_{i,j}$ is exponential with rate $1$ independent of the environmental process. The incoming fees are given by ${\boldsymbol r}=(1,1)$. We consider the case of initial reserves of ${\boldsymbol u}=(10,10)$ over a finite horizon of $T=50$. 

In Figure \ref{fig_Base} the diffusion approximation $\phi({\boldsymbol u},T)$ and single-switch approximation $\chi^j({\boldsymbol u},T)$, of the multivariate ruin probability are given as a function of the time horizon $T$ when starting in environmental state 1 at initiation. Due to the high environmental transition rates $Q$ and the relatively high claim arrival intensities, the diffusion approximation is close to the exact (simulated) multivariate ruin probability. 

To illustrate the effect of the correlation between the two business lines introduced by the environmental process, we also included the multivariate ruin probability assuming independence in the same figure. This probability is determined as the product of the individual ruin probabilities calculated using the univariate diffusion approximation \citep*[Eq. (3.11)]{Asmussen1989}. Due to the high transition probabilities of the environmental process and the high claim arrival intensities, the diffusion approximation outperforms all other approximations. The correlation between the approximative diffusion processes corresponding to the two business lines ($\bar{\Sigma}_{12}/\sqrt{\bar{\Sigma}_{11}\bar{\Sigma}_{22}}$, that is) is relatively low with 13\%; this explains why in this example there is a relatively modest difference between the (bivariate) diffusion approximation and the independent diffusion approximation.

\begin{figure}[ht]
\centering
\includegraphics[width=\textwidth]{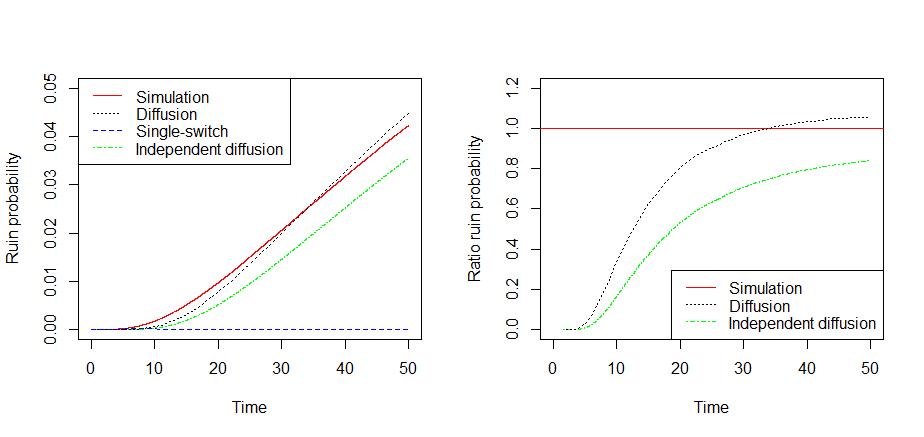}
\caption{Comparison of multivariate ruin probabilities as a function of time for the base parameter set. The right panel shows the value of the diffusion-based approximations relative to the true values.}
\label{fig_Base}
\end{figure}

As the claim size distribution does not depend on the state of the environment, the single-switch approximation is calculated using the results from Section \ref{subsec_special}. The performance of the Brownian motion scaling approach derived in Section \ref{subsec_BMscaling} on the single-switch univariate ruin probability $\psi_i^{j,k}({\boldsymbol u},\tau,T)$ is shown in Figure \ref{fig_BMscaling} as a function of $\tau$. 

\begin{figure}[ht]
\centering
\includegraphics[scale=0.6]{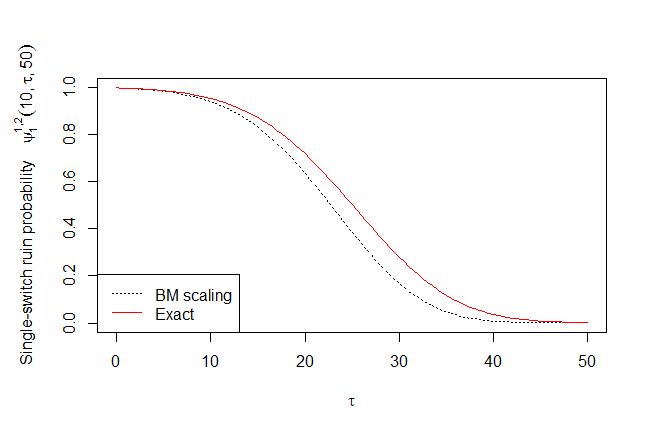}
\caption{Performance of BM scaling approach on the single switch univariate ruin probability $\psi_1^{1,2}(10,\tau,50)$ as a function of time $\tau$.}
\label{fig_BMscaling}
\end{figure}
\end{example}


\begin{example}(\textit{Impact environmental transition rates.})\label{ex_transition}
This example considers the impact of the transition rates of the environmental process on the multivariate ruin probability. In particular it shows that the accuracy of the single-switch approximation increases (with respect to the previous example) when the transition rates $q_{i,j}$ decrease (as expected). The same parameter settings are used as in the previous example, except that now the environmental transition rates are scaled by a factor $\frac{1}{64}$. Figure \ref{fig_LowEnvironment} shows that the decrease in transition rates indeed results in a more accurate single-switch approximation, independently of the remaining parameter values.  
When one decreases the probability of the environmental state making at most one transition, the performance of the single-switch approximation degrades.

\begin{figure}[ht]
\centering
\includegraphics[width=\textwidth]{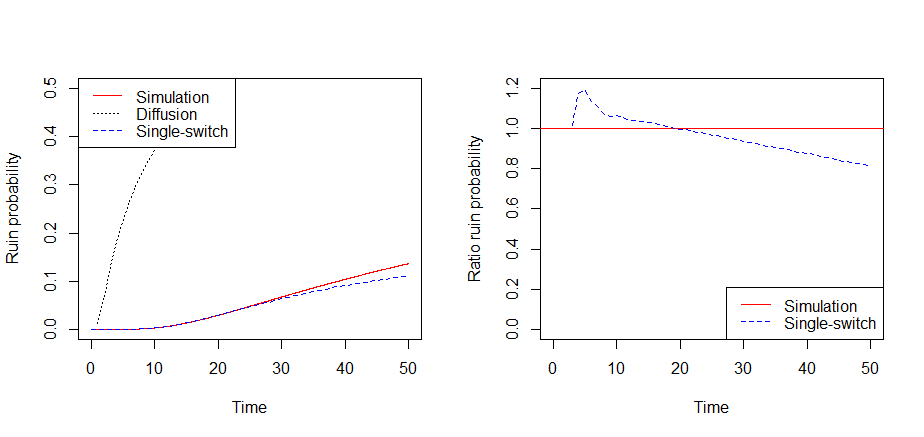}
\caption{Comparison of multivariate ruin probabilities as a function of time for low environmental transition rates $Q$. The right panel shows the value of the single-switch approximation relative to the true value.}
\label{fig_LowEnvironment}
\end{figure}
\end{example}

\begin{example}(\textit{Impact claim arrival intensities.})\label{ex_arrival}
This example studies the impact of the claim arrival intensity $\lambda$ on the multivariate ruin probability. The FCLT derived in Section \ref{subsec_FCLT} implies that the diffusion approximation of the risk process ${\boldsymbol X}$ is accurate when both the transition rates of the environmental process as well as the claim arrival intensities are high; this aligns with numerical findings reported in  \citet*{Asmussen1984}. When decreasing the claim arrival intensities to $\lambda_{1,1}=\lambda_{2,1}=0.36$ and $\lambda_{1,2}=\lambda_{2,2}=1.44$, Figure \ref{fig_LowArrival} shows that the diffusion approximation indeed loses accuracy. The effect on the single-switch approximation is negligible.

In this instance the multivariate ruin probability assuming independence is depicted in the same figure. This probability is determined as the product of the individual ruin probabilities calculated using the univariate diffusion approximation \citep*[Eq. (3.11)]{Asmussen1989}.

\vb

\begin{figure}[ht]
\centering
\includegraphics[scale=0.75]{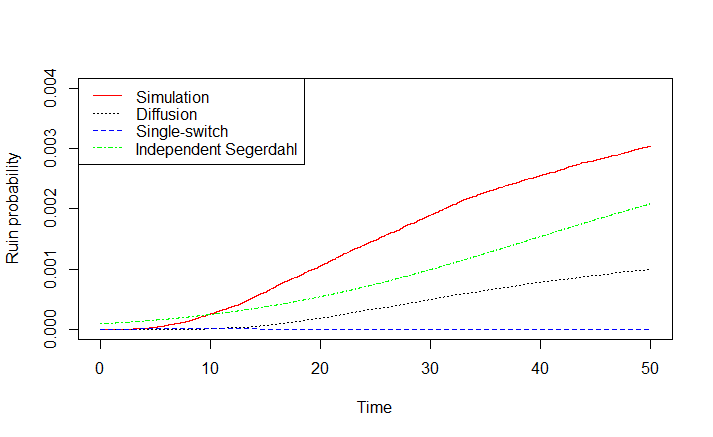}
\caption{Comparison of multivariate ruin probabilities as a function of time for low claim arrival intensities ${\boldsymbol \lambda}$.}
\label{fig_LowArrival}
\end{figure}

\end{example}
To summarize the numerical results obtained in this section,
\begin{itemize}
\item with {\it low} environmental transition rates, we recommend the use of the single-switch approximation;
\item with {\it high} environmental transition rates and {\it high} arrival intensity, we advise to use the diffusion approximation;
\item in case of {\it high} environmental transition rates and {\it low} arrival intensity neither the single-switch nor the diffusion approximation performs well, and the multivariate ruin probability is best approximated assuming independence between the components.

\end{itemize}

\section{Conclusion and outlook}\label{sec_conclusion}
This paper considered a multivariate risk process where the individual risk process of each business line is given by a Cram\'{e}r-Lundberg-type model. To model the dependence between different business lines, a Markov environmental process was introduced affecting each of the individual claims processes. By distinguishing between the transition speed of the environmental process being slow and fast, approximations for the multivariate ruin probability were developed. In case of low transition intensities, the environmental influence and the dependence between the different business lines disappears; this insight leads to the so-called single-switch approximation. For high transition intensities as well as high claim arrival intensities the centered and scaled multivariate risk process approaches a multivariate Brownian motion with drift, thus leading to a diffusion approximation. 
Numerical examples assessing the accuracy of these approximations were presented.

\vb

In case of low claim arrival intensities in a high environmental transition regime, both the single-switch approximation and the diffusion approximation do not perform well. This setting is marked for future research. Similarly, intermediate environmental transition rates require a different approach as the ones presented in this paper.

\bibliography{References}

\newpage
\begin{appendices}
\section{Proof of Theorem~\ref{Th_Pang_multi}}\label{appA}
In this section we follow the same line of reasoning as the proof of \citet*[Theorem 1.1]{Pang}. We begin by decomposing the diffusion-scaled process into three separate processes in Lemma \ref{lem_1}. The convergence of each process separately towards a Brownian motion is proven in Lemma \ref{lem_3}, Lemma \ref{lem_4} and Lemma \ref{lem_5}. These lemmas are the multivariate equivalents of Lemmas 2.6, 2.7 and 2.8 in \citet*{Pang}, respectively.
Finally, we conclude the proof with the joint convergence of the processes at the end of this section.
\begin{lemma}\label{lem_1}
The diffusion-scaled process ${\hat{\boldsymbol Y}}^n$ can be decomposed into the following three processes:
$$\hat{Y}_i^n(t)=\hat{U}_i^n(t)+\hat{V}_i^n(t)+\hat{W}_i^n(t)$$
where $$\hat{U}_i^n(t):=\frac{1}{n^\delta}\sum_{k=1}^{A^n(t)}\left(Z_{i,k}^n(J^n(\tau_{i,k}^n))-m^n_{i,J^n(\tau_{i,k}^n)}\right),$$
$$\hat{V}_i^n(t):=\frac{1}{n^\delta}\left(\sum_{k=1}^{A^n(t)}m^n_{i,J^n(\tau_{i,k}^n)}-\int_0^t m^n_{i,J^n(s)}\lambda^n_{i,J^n(s)}ds\right),$$
$$\hat{W}_i^n(t):=\frac{1}{n^\delta}\left(\int_0^t m^n_{i,J^n(s)}\lambda^n_{i,J^n(s)}ds-\sum_{j=1}^I\lambda_{i,j}^n\mu_{i,j}^n\pi_j t\right).$$
\end{lemma}
For each $n\in\mathbb{N}$, define $\mu^n_{i,*}:=\max_{j\in S}\mu_{i,j}^n$, $\lambda^n_{i,*}:=\max_{j\in S}\lambda_{i,j}^n$ and $\sigma^n_{i,*}:=\max_{j\in S}\sigma_{i,j}^n$. By the scaling of the parameters of $Y^n$, we obtain that, for all $i\in\{1,...,m\}$,
\begin{equation}\label{eq_asymp_max}
\frac{1}{n}\lambda^n_{i,*}\rightarrow \lambda_{i,*}, \ \mu^n_{i,*}\rightarrow \mu_{i,*} \ \text{and} \ \sigma^n_{i,*}\rightarrow \sigma_{i,*}, 
\end{equation}
in $\mathbb{R}$ as $n\rightarrow\infty$. Then we can find $n_1>0$ and $\Delta>0$ such that, for any $n>n_1$ and all $i\in\{1,...,m\}$,
\begin{equation}\label{eq_n1}
\max\Big\{\frac{1}{n}\lambda^n_{i,*},\mu^n_{i,*},\sigma^n_{i,*}\Big\}<\Delta.
\end{equation}

We fix the $n_1$ and $\Delta$ throughout the proof. We start by proving the convergence of ${\hat{\boldsymbol U}}^n$. For this we require the next auxiliary result, which is a direct extension of \citet*[Lemma 2.2]{Pang}.
\begin{lemma}\label{lem2.2}
Let $z_{1,1},z_{1,2},...,z_{n,n-1},z_{n,n}$ and $w_{1,1},w_{1,2},...,w_{n,n-1},w_{n,n}$ be complex numbers of modulus $\leq b$. Then 
$$\Bigg|\prod_{i=1}^m\prod_{j=1}^n z_{i,j}-\prod_{i=1}^m\prod_{j=1}^n w_{i,j}\Bigg|\leq b^{m-1}\sum_{i=1}^m\Bigg|\prod_{j=1}^n z_{i,j}-\prod_{j=1}^n w_{i,j}\Bigg|\leq b^{m+n-2}\sum_{i=1}^m\sum_{j=1}^n|z_{i,j}-w_{i,j}|$$
\end{lemma}

\begin{lemma}\label{lem_2}
The finite-dimensional distributions of $\hat{\boldsymbol U}^n=(\hat{U}_1^n,...,\hat{U}_m^n)$ converge to those of $\hat{\boldsymbol U}$, where $
\hat{\boldsymbol U}:=(\hat{U}_1,...,\hat{U_m})$ with $\hat{U}_1:=\{\hat{U}_1(t):t\geq 0\}$ being given by
\begin{equation}\label{eq_lem2}
\hat{\boldsymbol U}:=\begin{cases}\textbf{B}^1, &  \text{if \ } \delta=\tfrac{1}{2}, \ \alpha\geq 1,\\
\bm{0}, & \text{if \ } \delta=1-\tfrac{\alpha}{2}, \ \alpha\in(0,1);\end{cases}
\end{equation}
here $\textbf{B}^1=(B^1_1,...,B^1_m)$ is a $m$-dimensional 
Brownian motion with $\mathbb{E}\left[(\textbf{B}^1(t))(\textbf{B}^1(t))^\top\right]=\bar{\Sigma}^1t$, where $\bar{\Sigma}^1$ has been defined in Section \ref{subsec_FCLT}.
\end{lemma}
\begin{proof}
We need to prove 
\begin{equation}\label{eq_onedim}
(\hat{\boldsymbol U}^n(t_1),...,\hat{\boldsymbol U}^n(t_k))\Rightarrow(\hat{\boldsymbol U}(t_1),...,\hat{\boldsymbol U}(t_k)) \ \text{in} \ \mathbb{R}^{m\times k} \ \text{as} \ n\rightarrow\infty,
\end{equation}
for any $0\leq t_1\leq \cdots\leq t_k\leq T$ and $k\geq 1$.
We first consider the case of a single point in time: we aim at proving that, for each $t\geq 0$,
$$\hat{\boldsymbol U}^n(t)\Rightarrow\hat{\boldsymbol U}(t) \in \mathbb{R}^m \ \text{as} \ n\rightarrow\infty.$$
By L\'{e}vy's continuity theorem on $\mathbb{R}^m$ \citep*[see][Thm.\ 4.3]{kallenberg1997foundations}, it is sufficient to show convergence of the characteristic function: we have to prove that, as $n\to\infty$,\[
\psi_t^n(\bm{\theta}):=\mathbb{E}\left[e^{i\bm{\theta}^T\hat{\boldsymbol U}^n(t)}\right]\to \psi_t(\bm{\theta}):=\mathbb{E}\left[e^{i\bm{\theta}^T\hat{\boldsymbol U}(t)}\right]\] for every $\bm{\theta}\in\mathbb{R}^m$.
By the definition of $\hat{\boldsymbol U}$ in \eqref{eq_lem2}, \begin{equation}
\psi_t(\bm{\theta}):=\mathbb{E}\left[e^{i\bm{\theta}^T\hat{\boldsymbol U}(t)}\right]=\begin{cases}\exp\left(-\frac{1}{2}\bm{\theta}^T\bar{\Sigma}^1\bm{\theta} t\right), & \delta=\tfrac{1}{2}, \ \alpha\geq 1\\
1, & \delta=1-\tfrac{\alpha}{2}, \ \alpha\in(0,1).\end{cases}
\end{equation}
Let $\mathcal{A}^n_{t}:=\sigma\{\textbf{A}^n(s):0\leq s\leq t\}\vee\sigma\{J^n(s):0\leq s\leq t\}\vee \mathcal{N}$, where $\mathcal{N}$ is the collection of $P$-null sets. Then, by conditioning, we obtain

\begin{align}
\psi_t^n(\bm{\theta})&=\mathbb{E}\left[\exp\left(i\bm{\theta}^T\hat{\boldsymbol U}^n(t)\right)\right]=\mathbb{E}\left[\mathbb{E}\left[\exp\left(i\bm{\theta}^T\hat{\boldsymbol U}^n(t)\right)\big|\mathcal{A}^n_{t}\right]\right]\nonumber\\
&=\mathbb{E}\left[\prod_{i=1}^{m}\prod_{k=1}^{A_i^n(t)}\mathbb{E}\left[\exp\left(i\theta_i\frac{1}{n^\delta}\left(Z_{i,k}^n(J^n(\tau_{i,k}^n))-\mu^n_{i,J^n(\tau_{i,k}^n)}\right)\right)\big|\mathcal{A}^n_{t}\right]\right]\nonumber\\ 
&=\mathbb{E}\left[\prod_{i=1}^m\prod_{k=1}^{A_i^n(t)}\left(1-\frac{\theta_i^2}{2n^{2\delta}}(\sigma_{i,J^n(\tau_{i,k}^n)}^n)^2+o(n^{-2\delta})\right)\right]\label{eq_Taylor}
\end{align}
By \eqref{eq_asymp_max}, we can find $n_2$ such that for any $n>n_2$ and all $i\in\{1,...,m\}$, 
$$0<\max_{1\leq k\leq A^n_{i}(t)}\left\{\frac{\theta_i^2}{2n^{2\delta}}(\sigma_{i,J^n(\tau_{i,k}^n)}^n)^2-o\left(n^{-2\delta}\right)\right\}<1.$$
Furthermore, recall the definition of $n_1$ in \eqref{eq_n1}. Then, for $\delta=\tfrac{1}{2}$, $\alpha\geq 1$ and for any 
\begin{equation}\label{eq_n3}
n>n_3:=\max\{n_1,n_2\},
\end{equation}
we have
\begin{align}
\Big|\psi_t^n(\bm{\theta})-\psi_t(\bm{\theta})\Big|\leq&\:\:\mathbb{E}\Bigg[\Bigg|\prod_{i=1}^m\prod_{k=1}^{A_i^n(t)}\left(1-\frac{\theta_i^2}{2n}(\sigma_{i,J^n(\tau_{i,k}^n)}^n)^2+o(n^{-1})\right)
 -\prod_{i=1}^m\prod_{k=1}^{A_i^n(t)}\exp\left(-\frac{\theta_i^2}{2n}(\sigma_{i,J^n(\tau_{i,k}^n)}^n)^2\right)\Bigg|\Bigg]\nonumber\\
&+\Bigg|\mathbb{E}\left[\exp\left(-\sum_{i=1}^m\sum_{k=1}^{A_i^n(t)}\frac{\theta_i^2}{2n}(\sigma_{i,J^n(\tau_{i,k}^n)}^n)^2\right)\right]-\exp\left(-\sum_{i=1}^m\frac{\theta_i^2}{2}\bar{\sigma}_{i}^2\right)\Bigg|\nonumber\\
\leq &\:\: \mathbb{E}\left[\sum_{i=1}^m\sum_{k=1}^{A_i^n(t)}\frac{\theta_i^4}{4n}(\sigma_{i,J^n(\tau_{i,k}^n)}^n)^4\right]+o(1)\nonumber\\
&+\Bigg|\mathbb{E}\left[\exp\left(-\sum_{i=1}^m\sum_{k=1}^{A_i^n(t)}\frac{\theta_i^2}{2n}(\sigma_{i,J^n(\tau_{i,k}^n)}^n)^2\right)\right]-\exp\left(-\sum_{i=1}^m\frac{\theta_i^2}{2}\bar{\sigma}_{i}^2\right)\Bigg|\nonumber\\
\rightarrow& \ 0 \ \:\:\text{as} \ n\rightarrow\infty; \label{eq_triangle}
\end{align}
here, the first inequality is due to the triangle inequality and the second inequality follows by Lemma \ref{lem2.2} above in combination with \citet*[Lemma 2.3.]{Pang}. By \eqref{eq_n1}, for large enough $n$ defined above, we have
$$\mathbb{E}\left[\frac{1}{n}\sum_{k=1}^{A_i^n(t)}(\sigma_{i,J^n(\tau_{i,k}^n)}^n)^4\right]\leq\Delta^5 t, \ \forall i, \ t\geq 0.$$
As a result, the first two terms in the last equation converge to 0 when $n\rightarrow 0$. For the convergence of the last term, since the sequence $$\left\{\exp\left(-\sum_{i=1}^m\sum_{k=1}^{A_i^n(t)}\frac{\theta_i^2}{2n}(\sigma_{i,J^n(\tau_{i,k}^n)}^n)^2\right):n\geq 1\right\}$$ is bounded for each $t\geq 0$, it suffices to show that, for all $i\in\{1,...,m\}$,
\begin{equation}\label{eq_convR}
\sum_{k=1}^{A_i^n(t)}(\sigma_{i,J^n(\tau_{i,k}^n)}^n)^2\Rightarrow\bar{\sigma}_i^2 t,  \ \:\ \text{in} \ \mathbb{R} \ \text{as} \ n\rightarrow\infty.
\end{equation}
This follows from the convergences
$$\sum_{j=1}^I \frac{\lambda_{i,j}^n}{n}\int_0^t \mathbbm{1}(J^n(s)=j)\,{\rm d}s\rightarrow\sum_{j=1}^I\lambda_{i,j}\pi_{j}t \ \:{\rm a.s.},$$
and
$$\sum_{j=1}^I \frac{\lambda_{i,j}^n}{n}(\sigma_{i,j}^n)^2\int_0^t \mathbbm{1}(J^n(s)=j)\,{\rm d}s\rightarrow\sum_{j=1}^I\lambda_{i,j}\sigma_{i,j}^2\pi_{j}t \ \:{\rm a.s.}$$
by claim (4) in \citet*{Blom_queue}, the weak law of large numbers for Poisson processes, and the `random change of time lemma' \citep*[see][pp. 151]{billingsley1999convergence}.

For $\delta=1-\tfrac{\alpha}{2}$ and $\alpha\in(0,1)$, we follow the same line of reasoning and prove 
$$\Bigg|\mathbb{E}\left[\exp\left(-\sum_{i=1}^m\sum_{k=1}^{A_i^n(t)}\frac{\theta_i^2}{2n^{2\delta}}(\sigma_{i,J^n(\tau_{i,k}^n)}^n)^2\right)\right]-1\Bigg|\rightarrow 0 \ \: \text{as} \ n\rightarrow\infty.$$
Thereby, we have shown \eqref{eq_onedim}.

\vb

To show the convergence of the finite-dimensional distributions, it is sufficient to prove that for any 

$(\bm{\theta}^1,...,\bm{\theta}^k)\in\mathbb{R}^{m\times k}$ and $0\leq t_1<\cdots<t_l\leq T$,

$$\mathbb{E}\left[\exp\left(i\sum_{k=1}^l(\bm{\theta}^k)^\top\hat{\boldsymbol U}^{n}(t_k)\right)\right]\rightarrow\mathbb{E}\left[\exp\left(i\sum_{k=1}^l(\bm{\theta}^k)^\top\hat{\boldsymbol U}(t_k)\right)\right] \ \text{as} \ n\rightarrow\infty.$$
By the definition of $\hat{\boldsymbol U}$, we have
$$\mathbb{E}\left[\exp\left(i\sum_{k=1}^l(\bm{\theta}^k)^\top\hat{\boldsymbol U}(t_k)\right)\right]=
\begin{cases}{\displaystyle \exp\left(-\frac{1}{2}\sum_{k_1=1}^l\sum_{k_2=1}^l(\bm{\theta}^{k_1})^\top\bar{\Sigma}^1\bm{\theta}^{k_2} (t_{k_1}\wedge t_{k_2})\right)}, & \delta=\tfrac{1}{2}, \ \alpha\geq 1\\
1, & \delta=1-\tfrac{\alpha}{2}, \ \alpha\in(0,1).\end{cases}$$
By conditioning and direct calculation as in \eqref{eq_Taylor}, we have
\begin{align*}\mathbb{E}\bigg[\exp\bigg(i\sum_{k=1}^l(\bm{\theta}^k)^\top&\hat{\boldsymbol U}^{n}(t_k)\bigg)\bigg]=
\mathbb{E}\left[\prod_{j=1}^l\prod_{i=1}^{m}\exp\left(i\frac{1}{n^\delta}\sum_{k=j}^l\theta^k_i\sum_{h=A^n(t_{j-1})+1}^{A^n(t_j)}\left(Z_{i,h}^{n}(J^{n}(\tau_{i,h}^{n}))-\mu^{n}_{i,J^{n}(\tau_{i,h}^{n})}\right)\right)\right]\\
&\rightarrow\begin{cases}\prod_{j=1}^l\prod_{i=1}^{m}\exp\left(-\frac{1}{2}\left(\sum_{k=j}^l \theta^k_i\right)^2\bar{\sigma}_i^2(t_{j}-t_{j-1})\right), & \delta=\tfrac{1}{2}, \ \alpha\geq 1\\
1, & \delta=1-\tfrac{\alpha}{2}, \ \alpha\in(0,1),\end{cases}
\end{align*}
as $n\rightarrow\infty$, and
$$\prod_{j=1}^l\prod_{i=1}^{m}\exp\left(-\frac{1}{2}\left(\sum_{k=j}^l \theta^k_i\right)^2\bar{\sigma}_i^2(t_{j}-t_{j-1})\right)=\exp\left(-\frac{1}{2}\sum_{k_1=1}^l\sum_{k_2=1}^l(\bm{\theta}^{k_1})^\top\bar{\Sigma}^1\bm{\theta}^{k_2} (t_{k_1}\wedge t_{k_2})\right)$$
Applying L\'{e}vy's continuity theorem (on $\mathbb{R}^m$ now), the convergence can be shown in a similar way as in \eqref{eq_Taylor} and \eqref{eq_triangle}. Therefore, we have proven the weak convergence of the finite-dimensional distributions.
\end{proof} 

\begin{lemma}\label{lem_3}
$\hat{\boldsymbol U}^n\Rightarrow \hat{\boldsymbol U}$ in $\mathbb{D}^m$ as $n\rightarrow \infty$, where $\hat{\boldsymbol U}$ is given in \eqref{eq_lem2}.
\end{lemma}
\begin{proof}
Marginal tightness of the $\hat{\boldsymbol U}_i^n$ has been proven by \citet*[Lemma 2.5]{Pang} which implies joint tightness for $\hat{\boldsymbol U}$ \citep*[Lemma 7.14(i)]{kosorok2008introduction}. 
Together with the continuity of $\hat{\boldsymbol U}$, Lemma \ref{lem_2}, we apply \citet*[Thm. 13.1]{billingsley1999convergence} to conclude the convergence of $\hat{\boldsymbol U}^n$.
\end{proof}

\begin{lemma}\label{lem_4}
$\hat{\boldsymbol V}^n\Rightarrow \hat{\boldsymbol V}$ in $\mathbb{D}^m$ as $n\rightarrow \infty$, where $\hat{\boldsymbol V}$ is given by
\begin{equation}\label{eq_lem4}
\hat{\boldsymbol V}:=\begin{cases} \textbf{B}^2, &  \text{if \ } \delta=\tfrac{1}{2}, \ \alpha\geq 1,\\
\mathbf{0}, & \text{if \ } \delta=1-\tfrac{\alpha}{2}, \ \alpha\in(0,1);\end{cases}
\end{equation}
here $\textbf{B}^2:=(B_1^2,...,B_m^2)$ is a $m$-dimensional zero-mean Brownian motion with $\mathbb{E}\left[(\textbf{B}^2(t))(\textbf{B}^2(t))^T\right]=\bar{\Sigma}^2t$, where $\bar{\Sigma}^2$ has been defined in Section \ref{subsec_FCLT}.
\end{lemma}
\begin{proof}
As centered Poisson processes are $\mathbb{R}$-valued martingales, i.e., for each $n\in\mathbb{N}$ and every $i\in\{1,...,m\}$ the process $$\left\{A_i^n(t)-\int_0^t\lambda_{i,J^n(u)}^n \,{\rm d}u:t\geq 0\right\}$$ is a martingale, $\hat{\boldsymbol V}^n$ is a $\mathbb{R}^m$-valued martingale. The maximum jump for $\hat{V}_i^n$ is $\mu_{i,*}^n/n^\delta$. By \eqref{eq_asymp_max}, we obtain that the expected value of the maximum jump is asymptotically negligible, i.e., for all $i\in\{1,...,m\}$
$$\frac{1}{n^\delta}\mathbb{E}\left[\mu_{i,*}^n\right]\rightarrow 0, \ \: \text{as} \ n\rightarrow\infty.$$
For $n\in\mathbb{N}$, let $\{[\hat{V}_i^n,\hat{V}_j^n](t):t\geq 0\}$ be the quadratic covariation process of $\hat{V}_i^n$ and $\hat{V}_i^n$. Then, for each $t$, we have by the quadratic variation of a compound Poisson process, as $n\to\infty$,
\begin{align}
[\hat{V}_i^n,\hat{V}_j^n](t)&=\frac{1}{n^{2\delta}}\begin{cases}
\sum_{k=1}^{A_i^n}(\mu^n_{i,J^n(\tau_{i,k})})^2 & \text{for} \ i=j,\\
0, & \text{for} \ i\neq j,
\end{cases}\\&\Rightarrow\begin{cases}\bar{\Sigma}_{i,j}^2 t, &  \text{if \ } \delta=\tfrac{1}{2}, \ \alpha\geq 1,\\
0, & \text{if \ } \delta=1-\tfrac{\alpha}{2}, \ \alpha\in(0,1),\end{cases}  \ 
\end{align}
in $ \mathbb{R}^m$,
where the convergence is proven in the same way as \eqref{eq_convR}. Applying \citet*[Thm. 2.1]{whitt2007}, we have shown the convergence of $\hat{\boldsymbol V}^n$.
\end{proof}

\begin{lemma}\label{lem_5}
$\hat{\boldsymbol W}^n\Rightarrow \hat{\boldsymbol W}$ in $\mathbb{D}^m$ as $n\rightarrow \infty$, where the limit process $\hat{\boldsymbol W}:=\{\hat{\boldsymbol W}(t):t\geq 0\}$ is given by
\begin{equation}\label{eq_lem5}
\hat{\boldsymbol W}:=\begin{cases}\mathbf{0}, &  \text{if \ } \delta=\tfrac{1}{2}, \ \alpha> 1,\\
\textbf{B}^3, & \text{if \ } \delta=1-\tfrac{\alpha}{2}, \ \alpha\in(0,1];\end{cases}
\end{equation}
here $\textbf{B}^3=(B_1^3,...,B_m^3)$ is a $m$-dimensional Brownian motion with $\mathbb{E}\left[(\textbf{B}^3(t))(\textbf{B}^3(t))^T\right]=\bar{\Sigma}^3t$, where $\bar{\Sigma}^3$ has been defined in Section \ref{subsec_FCLT}.
\end{lemma}
\begin{proof}
Let $\bar{{\boldsymbol W}}^n:=(\bar{W}_1^n,...,\bar{W}_m^n)$ with, for $i=1,\ldots,m,$
$$\bar{W}_i^n:=\frac{1}{n^\delta}\left(\sum_{k=1}^I\mu_{i,k}^n\lambda_{i,k}^n\int_0^t \mathbbm{1}(J^n(s)=k)\,{\rm d}s-\sum_{k=1}^I\mu_{i,k}^n\lambda_{i,k}^n\pi_k t \right).$$
By \citet*[Prop. 3.2]{Blom_queue}, we have, as $n\rightarrow\infty$,
$$\bar{{\boldsymbol W}}^n\Rightarrow\begin{cases}\mathbf{0}, &  \text{if \ } \delta=\tfrac{1}{2}, \ \alpha> 1,\\
{\boldsymbol B}^3, & \text{if \ } \delta=1-\tfrac{\alpha}{2}, \ \alpha\in(0,1],\end{cases}$$
with ${\boldsymbol B}^3$ as defined above. This concludes the proof.
\end{proof}

\begin{proof}[Proof of Theorem~\ref{Th_Pang_multi}]
By Lemmas \ref{lem_1}, \ref{lem_3}, \ref{lem_4} and \ref{lem_5}, we have obtained the marginal convergence of $\hat{\boldsymbol U}^n\Rightarrow \hat{\boldsymbol U}$, $\hat{\boldsymbol V}^n\Rightarrow \hat{\boldsymbol V}$ and $\hat{\boldsymbol W}^n\Rightarrow \hat{\boldsymbol W}$. We now have to prove the joint convergence $$\left(\hat{\boldsymbol U}^n,\hat{\boldsymbol V}^n,\hat{\boldsymbol W}^n\right)\Rightarrow \left(\hat{\boldsymbol U},\hat{\boldsymbol V},\hat{\boldsymbol W}\right)$$ where $\hat{\boldsymbol U}$, $\hat{\boldsymbol V}$ and $\hat{\boldsymbol W}$ are mutually independent.
To this end, we first note that $\hat{\boldsymbol U}^n$ and $\hat{\boldsymbol V}^n$ are compensated compound Poisson processes which are martingales. Furthermore, by \citet*[Lem. 3.1]{Blom_queue} the process $\hat{\boldsymbol W}^n$ is a martingale.

By \citet*[Thm.\ 3.12, Ch.\ VIII]{jacod2002limit} (or a slightly less extensive version is given in \citet*[Cor. 2.17, pp. 264]{aldous1985ecole}), it suffices to show that, for ${\boldsymbol M}^n:=(\hat{\boldsymbol U}^n,\hat{\boldsymbol V}^n,\hat{\boldsymbol W}^n)$ and ${\boldsymbol M}:=(\hat{\boldsymbol U},\hat{\boldsymbol V},\hat{\boldsymbol W})$,
\begin{equation}\label{eq_convmart}
\left[{\boldsymbol M}^n,{\boldsymbol M}^n\right](t)\Rightarrow \begin{bmatrix}
    \hat{\Sigma}^1 & \textbf{0} & \textbf{0} \\
    \textbf{0} & \hat{\Sigma}^2 & \textbf{0} \\
    \textbf{0} & \textbf{0} & \hat{\Sigma}^3
  \end{bmatrix} t,
\end{equation}
where
$$(\hat{\Sigma}^1,\hat{\Sigma}^2)=\begin{cases}(\bar{\Sigma}^1,\bar{\Sigma}^2), &  \text{if \ } \delta=\tfrac{1}{2}, \ \alpha\geq 1,\\
(\mathbf{0},\mathbf{0}), & \text{if \ } \delta=1-\tfrac{\alpha}{2}, \ \alpha\in(0,1),\end{cases}\:\:\: \hat{\Sigma}^3=\begin{cases}\mathbf{0}, &  \text{if \ } \delta=\tfrac{1}{2}, \ \alpha> 1,\\
\bar{\Sigma}^3, & \text{if \ } \delta=1-\tfrac{\alpha}{2}, \ \alpha\in(0,1].\end{cases}$$

For $\alpha\in(0,\infty)$ and $\delta\in[\tfrac{1}{2},1)$,  
\begin{align*}
\left[\hat{U}_i^n,\hat{V}_j^n\right](t)&=\frac{1}{n^{2\delta}}\begin{cases}
\sum_{k=1}^{A_i^n}\left(Z_{i,k}^n(J^n(\tau_{i,k}^n))-\mu^n_{i,J^n(\tau_{i,k}^n)}\right)\mu^n_{i,J^n(\tau_{i,k})} & \text{for} \ i=j,\\0, & \text{for} \ i\neq j,\end{cases}
\\ &\Rightarrow 0 \ \text{in} \ \mathbb{R}^m, \ \text{as} \ n\rightarrow\infty.
\end{align*}
Together with $\big[\hat{U}_i^n,\hat{W}_j^n\big](t)=0$ and $\big[\hat{V}_i^n,\hat{W}_j^n\big](t)=0$ this proves \eqref{eq_convmart}.
The proof of Theorem \ref{Th_Pang_multi} is completed by applying the continuous mapping theorem.
\end{proof}

\end{appendices}

\end{document}